\documentclass{article}
\usepackage{amssymb,amsmath,amsthm,graphicx}

\def\qed{\hfill {\hbox{${\vcenter{\vbox{               %HOLLOW SQUARE
   \hrule height 0.4pt\hbox{\vrule width 0.4pt height 6pt
   \kern5pt\vrule width 0.4pt}\hrule height 0.4pt}}}$}}}

\def\utr{\, \underline{\triangleright}\, }
\def\otr{\, \overline{\triangleright}\, }

\newtheorem{theorem}{Theorem}

\newtheorem{proposition}[theorem]{Proposition}
\newtheorem{corollary}[theorem]{Corollary}

\theoremstyle{definition}
\newtheorem{example}{Example}
\newtheorem{definition}{Definition}
\newtheorem{remark}{Remark}

\date{}

\title{\Large \textbf{Birack Bracket Quivers and Framed Links}}

\author{Sam Nelson\footnote{Email: Sam.Nelson@cmc.edu. Partially supported by Simons Foundation collaboration grant 702597}\and
Haoqi (Tom) Tang\footnote{Email: htot2022@mymail.pomona.edu}}

\begin{document}
\maketitle

\begin{abstract}
We introduce \textit{birack brackets}, skein invariants of birack-colored
framed classical and virtual knots and links with values in a commutative 
unital ring. The multiset of birack bracket values over the homset from a 
framed link's fundamental birack then forms an invariant of framed links. 
We then categorify this multiset to define a quiver-valued invariant of 
framed knots and links. From this quiver we define new polynomial
invariants of framed knots and links.
\end{abstract}

\parbox{5.5in} {\textsc{Keywords:} Biracks, birack brackets, 
quantum enhancements of counting invariants, framed knots and links

\smallskip

\textsc{2020 MSC:} 57K12}

\section{\large\textbf{Introduction}}\label{I}

\textit{Biracks} are algebraic structures whose axioms encode the
framed version of the Reidemeister moves, first introduced in \cite{FRS}
and later studied in works such as \cite{N,NW} etc. Biracks satisfying a 
stronger axiom coming from the unframed Reidemeister I move
known as \textit{biquandles}, notable for their utility in defining
invariants of unframed knots and links, have been studied in many works
such as \cite{EN, FJK, NOR} and more.

In particular, the \textit{birack homset invariant} is the set of
birack homomorphisms from the fundamental birack of an framed oriented 
classical or virtual knot or link $L$ to a (usually finite) birack $X$. 
The elements of this homset invariant can be represented as \textit{birack 
colorings} of a diagram $D$ of $L$, i.e., assignments of elements of $X$
to the semiarcs in a diagram of $L$ satisfying a certain coloring condition
at each crossing. 

Invariants of birack-colored diagrams provide \textit{enhancements} of the 
birack counting invariant, generally stronger invariants of framed oriented 
knots and links which determine the birack counting invariant. In some cases
these can be summed over a finite set of framings of a knot or link to obtain
invariants of unframed oriented knots and links; see \cite{BN, CEGN,NW} for
examples.

\textit{Biquandle brackets} are skein invariants of biquandle-colored 
diagrams with skein coefficients depending on the colors at a crossing,
first introduced in \cite{NOR} and later studied in works such as 
\cite{GNO,HVW,NO} etc. The biquandle bracket invariant was categorified
to the case of \textit{biquandle bracket quivers} in \cite{FN} by
applying a construction from \cite{CN}.

In this paper we generalize the biquandle bracket construction to the case of 
framed oriented classical or virtual knots and links, obtaining \textit{birack 
brackets} and their categorifications, \textit{birack bracket quivers}. The 
paper is organized as follows. In Section \ref{RB} we review the basics of
biracks and the birack homset for framed links. In Section \ref{BB} we 
introduce the birack bracket construction along with a resulting
polynomial invariant of framed links. In Section \ref{BBQ} we categorify
the birack bracket invariant to obtain birack bracket quivers and use 
these to obtain further polynomial invariants of framed knots and links.  
We collect some examples and computations, and we conclude
in Section \ref{Q} with some questions for future research.

This paper, including all text, illustrations, and python code for 
computations, was written strictly by the authors without the use of 
generative AI in any form.

\section{\large\textbf{Birack Review}}\label{RB}

As mentioned in the Introduction, to define knot invariants
various mathematical structures have been implemented to encode the crossing 
information of a knot, taking the forms of groups, tensors, and 
more abstract algebraic structures. Reidemeister moves characterize 
the equivalence relation of ambient isotopy on diagrams; the corresponding
algebraic structures in result in different knot invariants. 
More specifically, \textit{quandles} represent the algebraic structure of
Reidemeister moves on arcs in classical knots; \textit{biquandles}  
represent Reidemeister moves on semi-arcs in classical and virtual knots, and 
\textit{racks} represent the algebraic structure of arcs in framed knots and 
links through replacing Reidemeister I move with the framed version of the 
move. In this paper, we consider the case of \textit{biracks}, algebraic 
structures generated by semiarcs in a link diagram with axioms corresponding to 
blackboard framed isotopy.

Recall that two blackboard framed oriented knot or link diagram represent
isotopic framed knots or links if and only if they are related by a sequence
of the \textit{blackboard framed Reidemeister moves}, of which a generating
set is:
\[\includegraphics{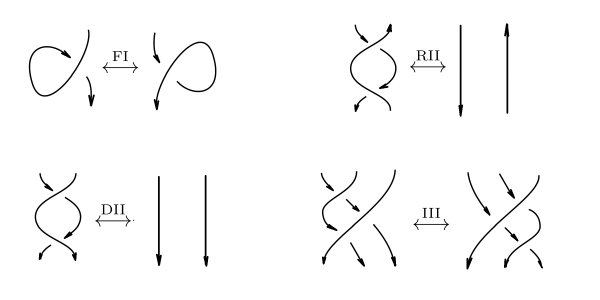}\]

Next we recall a definition from \cite{N}, reformulated in our current notation.

\begin{definition}
A \textit{framed birack} is a set $X$ with two binary operations $\utr,\otr$
satisfying the axioms
\begin{itemize}
\item[(i)] For all $x\in X$ we have
\[(x\utr x)\otr(x\otr x)=(x\otr x)\utr(x\utr x),\]
\item[(ii)] The maps $\alpha_x,\beta_x:X\to X$ and $S:X\times X\to X\times X$
defined by $\alpha_x(y)=y\otr x$, $\beta_y(x)=x\utr y$ and 
$S(x,y)=(y\otr x,x\utr y)$ are invertible for all $x\in X$ and
\item[(iii)] The \textit{exchange laws}
\[\begin{array}{rcl}
(x\utr y)\utr (z\utr y) & = & (x\utr z)\utr (y\otr z) \\ 
(x\utr y)\otr (z\utr y) & = & (x\otr z)\utr (y\otr z) \\ 
(x\otr y)\otr (z\otr y) & = & (x\otr z)\otr (y\utr z) \\ 
\end{array}\]
are satisfied for all $x,y,z\in X$.
\end{itemize}
A birack in which we also have $x\utr x=x\otr x$ for all $x\in X$
is a \textit{biquandle}. A biquandle in which $x\otr y=x$ for all $x,y\in X$
is a \textit{quandle}. 
\end{definition}

\begin{remark}
Biracks are sometimes defined without requiring our axiom (i); we will 
refer to these as \textit{regular biracks}. Essentially, regular biracks
can be used to define invariants of regular isotopy, e.g. braid invariants,
while framed biracks provide invariants of framed knots and links. Unless 
otherwise specified, we will mean ``framed birack'' whenever we say ``birack''.
\end{remark}

As observed in \cite{N}, for any finite birack $X$ there is a unique 
bijection $\pi:X\to X$ called the \textit{kink map} satisfying
\[x\utr \pi(x)=\pi(x)\otr x\]
for all $x\in X$. A biquandle is a birack whose kink map is the identity.

More precisely, the framed birack axioms are chosen so that for any valid 
birack coloring of an oriented framed link diagram on one side of a framed 
Reidemeister move, there is a unique valid coloring of the diagram on the 
other side of the move which agrees outside the neighborhood of the move.
Then starting with a single strand, using Reidemeister II and III moves 
we can perform what's sometimes called the \textit{Fenn-Rourke move} (FR):
\[\includegraphics{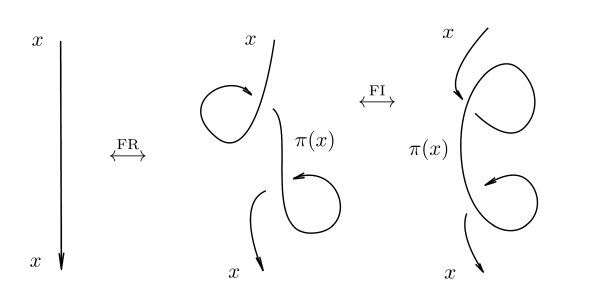}.\]
The bijectivity of colorings establishes the bijectivity of the kink map $\pi$;
the framed Reidemeister I move condition expressed algebraically by axiom (i)
then shows that the kink map (and hence its inverse) does not depend on which
side of the strand the loop is on.

\begin{example}
One infinite family of examples of biracks is the set of 
$(t,s,r)$-\textit{biracks}. Let $X$ be a $\widetilde{\Lambda}$-module 
where $\widetilde{\Lambda} = \mathbb{Z}[t^{\pm 1},s,r^{\pm 1}]/
(s^2-s(r-t))$. Then $X$ is a birack under the operations
\[x \otr y = r x, \quad x \utr y = tx +sy.\]
To see this we can verify the axioms:
 
Checking axiom (i), we have
\begin{eqnarray*}
(x\utr x)\otr(x\otr x)
& = & r(tx+sx) \\
& = & r(t+s)x \\
& = & (tr+st+s(r-t))x \\
& = & (tr+st+s^2)x \\
& = & t(rx)+s(tx+sx) \\
& = & (x\otr x)\utr(x\utr x)
\end{eqnarray*}
as required.

Axiom (ii) says that $S(x,y) = (y\otr x, x\utr y) = (ry, tx +sy)$ is 
invertible for a birack; namely, there exists $S^{-1}$ such that 
$S^{-1} (y \otr x, x\utr y) = (x,y)$, or $S^{-1}(ry, tx+sy) = (x,y)$. 
Then we note that $S^{-1}$ is given by 
$S^{-1}(x, y) = (t^{-1}(y-sr^{-1}x), r^{-1}x)$, where $s$ is not necessarily 
invertible in $\widetilde{\Lambda}$. 
 
Finally, checking the the exchange laws of axiom (iii):
\begin{eqnarray*}
 (x\utr y)\utr (z\utr y) &= &(tx+sy) \utr (tz +sy) \\
 & = &  t^{2}x +sty+ stz +s^{2}y \\
 & = & = t^{2}x +tsz +rsy  \\
 & = &  (x\utr z)\utr (y\otr z) 
\end{eqnarray*}
since $s^{2}+st=rs$,
\begin{eqnarray*}
 (x\utr y)\otr (z\utr y) 
 & = & (tx+sy)\otr (tz+sy) \\
 & = &  r(tx +sy)\\
 & = & rx \utr ry \\
 & = &  (x\otr z)\utr (y\otr z) 
 \end{eqnarray*}
and
\begin{eqnarray*}
 (x\otr y)\otr (z\otr y) & = & rx \otr rz \\
 & = & r^{2} x \\
 & = & rx \otr (ty +sz) \\ 
 & = & (x\otr z)\otr (y\utr z) 
 \end{eqnarray*}

We note that the kink map in an Alexander birack is given by
$\pi(x)=(r-s)t^{-1}x$; this is bijective since $(r-s)$ is invertible in 
$\widetilde{\Lambda}$ with inverse $t^{-1}r^{-1}(s+t)$.
Hence, any $\widetilde{\Lambda}$-module is an example of a birack with
these operations.
\end{example}

\begin{example}
One concrete example of a $(t,s,r)$-birack is $(t=1, s=2)$ on 
$\mathbb{Z}_{4}$. By the condition $s^2=s(r-t)$, 
$r$ is determined to be either 3 or 1; let us choose $r=1$. 
The operation tables for $\utr, \otr$ are then given below:
\[
    \begin{array}{c|cccc}
      \utr & 0 & 1 & 2 & 3\\
      \hline
      0 & 0 & 2 & 0 & 2 \\
      1 & 1 & 3 & 1 & 3 \\
      2 & 2 & 0 & 2 & 0 \\
      3 & 3 & 1 & 2 & 1 \\
\end{array}
\quad
    \begin{array}{c|cccc}
      \otr & 0 & 1 & 2 & 3\\
      \hline
      0 & 0 & 0 & 0 & 0 \\
      1 & 1 & 1 & 1 & 1 \\
      2 & 2 & 2 & 2 & 2\\
      3 & 3 & 3 & 3 & 3 
\end{array}.\]
The kink map in this case is $\pi(x)=3x$.
\end{example}

\begin{definition}   
Let $X$ and $Y$ be biracks with operations $\utr_{X},\otr_{X},\utr_{Y},\otr_{Y}$.
A \textit{birack homomorphism} is a function $f:X\to Y$ such that for 
all $x,y\in X$ we have
\[ f(x\utr_{X}\, y) = f(x)\utr_{Y} f(y)\quad \mathrm{and}\quad
f(x\otr_{X}\, y) = f(x)\otr_{Y} f(y).\]
\end{definition}

\begin{definition} \label{def:coloring}
A \textit{birack coloring} of a framed oriented knot or link diagram is an 
assignment of elements of our birack $X$ to the semiarcs of our diagram such
that at every crossing we have
\[\scalebox{0.7}{\includegraphics{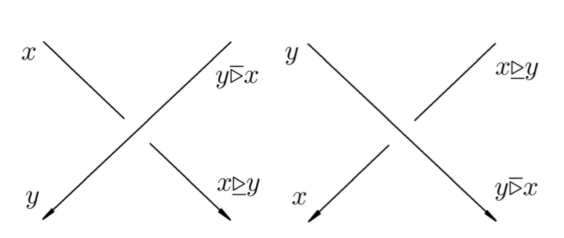}}.\]
\end{definition}

%The birack axioms are chosen so that for every birack coloring of a framed
%oriented link diagram before a framed Reidemeister move, there is
%a unique coloring of the diagram after the move that agrees with the original
%coloring outside the neighborhood of the move.

\begin{definition}
Let $L$ be a framed oriented knot or link diagram. The \textit{fundamental
birack} of $L$, denoted $\mathcal{B}(L)$, has a presentation with generators
corresponding to the semiarcs in $L$ and relations between the generators at
crossings as depicted in Definition \ref{def:coloring}. The elements of
$\mathcal{B}(L)$ are then equivalence classes of birack words in these 
generators (including symbols such as $\alpha_x^{-1}(y)$ etc. as required
by the axioms) modulo the equivalence relation generated by the crossing 
relations and the biquandle axioms.
\end{definition} 
       
A birack coloring of a framed knot diagram $L$ by a finite birack $X$ 
determines a birack homomorphism $f:\mathcal{B}(L)\to X$ by assigning
an image to each generator consistently with the birack axioms. That is,
a birack homomorphism $f$ assigns to each semi-arc in $\mathcal{B}(L)$ 
an image element in $X$ so that the crossing information of semi-arcs in 
$\mathcal{B}(L)$ satisfies the algebraic relations of the birack $X$. 
One framed knot birack homomorphism represents one coloring of 
the framed link $L$. The set of all framed knot birack homomorphisms from
the fundamental birack of $L$ to a birack $X$, known as the \textit{birack 
homset}, can then be identified with the set of all possible colorings 
of $K$, denoted by $\mathrm{Hom}\mathcal({B}(L), X)$.

\begin{definition}
The \textit{birack counting invariant} of a framed oriented knot or link 
$L$ is defined as the cardinality of the homset, 
\[\Phi_{X}^{\mathbb{Z}} = |\mathrm{Hom}(\mathcal{B}(L), X)|.\]
\end{definition}

\begin{example}\label{ex:31homset}
The fundamental birack of the framed knot $3_1$ can be computed by labeling 
semi-arcs of the knot as generators and crossing information as relations:
\[\includegraphics{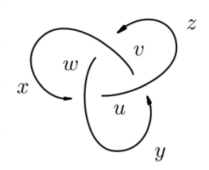}\]
\[
\mathcal{B}(3_1)
=\bigl\langle\,x,y,z,u,v,w \;\bigm|\;
x\utr y = u,\;
y\otr x = w,\;
y \utr z = v,\;
z \otr y = u,\;
z\utr x = w,\;
x \otr z = v
\bigr\rangle.
\]
Then let $X$ be the birack given by the operation tables
\[
\begin{array}{c|c c c}
      \utr & 1 & 2 & 3 \\
      \hline
      1 & 2 & 2 & 2 \\
      2 & 3 & 3 & 3 \\
      3 & 1 & 1 & 1       
    \end{array}
\quad
    \begin{array}{c|c c c}
      \otr & 1 & 2 & 3 \\
      \hline
      1 & 3 & 3 & 3 \\
      2 & 1 & 1 & 1 \\
      3 & 2 & 2 & 2 \\
\end{array}
\]
All possible homomorphisms from the fundamental birack $\mathcal{B}(3_1)$ 
to the birack $X$ form the homset $\mathrm{Hom} (\mathcal{B}(3_1), X)$. 
There are three $X$-colorings:
\[\includegraphics{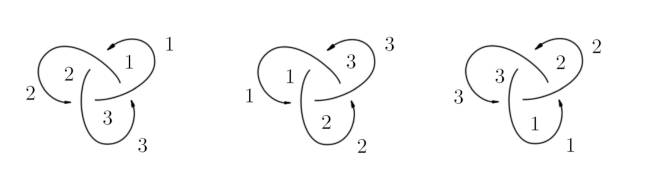}\]
Therefore, the birack counting invariant $|\mathrm{Hom}(\mathcal{B}(3_{1}), X)|$ is 3. 
\end{example}

\begin{remark}
In \cite{N} the birack counting invariant for framed oriented knots and links
was used to define an invariant of unframed oriented knots and links by 
collecting the counting invariant values of a complete period of framings 
modulo the \textit{birack rank}, i.e., the smallest integer $n\ge 1$ such that
$\pi^n=\mathrm{Id}$.
\end{remark}

\section{\large\textbf{Birack Brackets}}\label{BB}

We now state our main definition.

%One class of strong invariants are formulated as the knot polynomials derived via smoothing of knots and links using skein relations. The skein relation for birack follows its two axioms. 

\begin{definition}
Let $X$ be a birack and $R$ a commutative unital ring. A birack bracket is 
defined by two maps $A,B:X\times X\to R^{\times}$ from ordered pairs of elements
of $X$ to units in $R$ denoted 
$A(x,y) = A_{x,y}$ and $B(x,y) = B_{x,y}$ such that
\begin{itemize}
\item[(i)] For all $x\in X$ we have
\[A_{x,x}^2B_{x,x}^{-1}=A_{x\otr x,x\utr x}^2B_{x\otr x,x\utr x}^{-1},\]
\item[(ii)] For all $x,y\in X$,
\[\delta=-A_{x,y}B_{x,y}^{-1}-A_{x,y}^{-1}B_{x,y}\ne 0\]
and 
\item[(iii)] For all $x,y,z\in X$, 
\[\begin{array}{rcl}
A_{x,y}A_{y,z}A_{x\utr y,z\otr y} & = & A_{x,z}A_{y\otr x,z\otr x}A_{x\utr z,y\utr z} \\
A_{x,y}B_{y,z}B_{x\utr y,z\otr y} & = & B_{x,z}B_{y\otr x,z\otr x}A_{x\utr z,y\utr z} \\
B_{x,y}A_{y,z}B_{x\utr y,z\otr y} & = & B_{x,z}A_{y\otr x,z\otr x}B_{x\utr z,y\utr z} \\
A_{x,y}A_{y,z}B_{x\utr y,z\otr y} & = & 
A_{x,z}B_{y\otr x,z\otr x}A_{x\utr z,y\utr z} 
+A_{x,z}A_{y\otr x,z\otr x}B_{x\utr z,y\utr z} \\ 
& & +\delta A_{x,z}B_{y\otr x,z\otr x}B_{x\utr z,y\utr z} 
+B_{x,z}B_{y\otr x,z\otr x}B_{x\utr z,y\utr z} \\
B_{x,y}A_{y,z}A_{x\utr y,z\otr y} 
+A_{x,y}B_{y,z}A_{x\utr y,z\otr y} & & \\
+\delta B_{x,y}B_{y,z}A_{x\utr y,z\otr y} 
+B_{x,y}B_{y,z}B_{x\utr y,z\otr y}  
& = & B_{x,z}A_{y\otr x,z\otr x}A_{x\utr z,y\utr z} \\
\end{array}\]
\end{itemize}
where $A(x,y)$ and $B(x,y)$ are denoted $A_{x,y}$ and $B_{x,y}$ and
$A_{x,x}^2B_{x,x}^{-1}$ is denoted by $w_x$. A birack in which $w_x=w_y$
for all $x,y\in X$ is called \textit{homogeneous}.
\end{definition}

The maps $A,B:X\times X\to R^{\times}$ give the smoothing coefficients 
using the \textit{birack skein relations}: 
\[\includegraphics{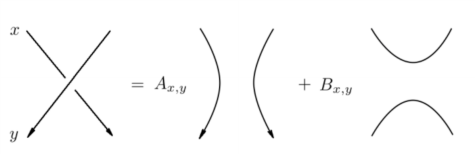}\ \raisebox{0.425in}{and}\
\includegraphics{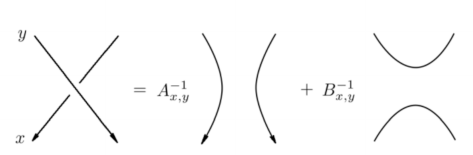}.
\]

The birack bracket axioms are then the conditions required for the 
\textit{state-sum} value to be unchanged by framed Reidemeister moves.
More precisely, simultaneously performing all smoothings results in a
sum of products of skein coefficients times diagrams without crossings, i.e.
disjoint closed curves, known as \textit{Kauffman states}. The value of a
state is $\delta^k$ where $k$ is the number of disjoint curves in the state
times the associated product of skein coefficients. Indeed, we have:

\begin{proposition}
Let $L$ be an oriented link diagram, $X$ a finite birack and $\beta$ a birack
bracket with coefficients in a commutative unital ring $R$. Then for each
coloring $f\in \mathrm{Hom}(\mathcal{BR}(L),X)$, the state-sum value
\[\beta(f)
=\sum_{\mathrm{Kauffman \ states}} 
\left(\prod_{\mathrm{smoothings}}C_{xy}^{\pm 1}\delta^k\right)\]
defined by summing over all Kauffman states the product of the state's 
smoothing coefficients $C_{xy}^{\pm 1}$ times $\delta$ to the power of the 
number of components in the state 
is invariant under $X$-colored framed Reidemeister moves.
\end{proposition}

\begin{proof}
Invariance under the direct and reverse type II moves and the all-positive
type III move is the same as in \cite{NOR}; let us consider the framed type
I move. 
\[\includegraphics{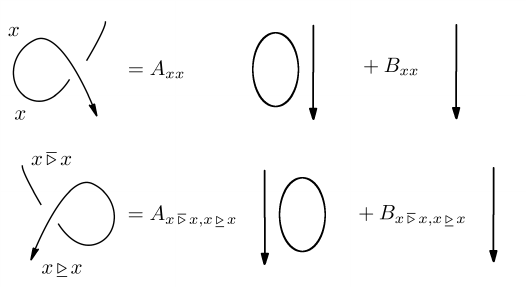}\]
Comparing the two sides of the move we obtain the requirement that
\[A_{xx}\delta+B_{xx}=A_{x\otr x,x\utr x}\delta+B_{x\otr x,x\utr x}\]
which says
\[A_{xx}(-A_{xx}B_{xx}^{-1}-A_{xx}^{-1}B_{xx})+B_{xx}
=A_{x\otr x,x\utr x}(-A_{x\otr x,x\utr x}B_{x\otr x,x\utr x}^{-1}
-A_{x\otr x,x\utr x}^{-1}B_{x\otr x,x\utr x})+B_{x\otr x,x\utr x}\]
which reduces to 
\[A_{xx}^2B_{xx}^{-1}=A_{x\otr x,x\utr x}^2B_{x\otr x, x\utr x}^{-1}.\]
\end{proof}

We can specify a birack bracket for a finite biquandle $X$
with values in $R$ with a block matrix $[A|B]$ whose entries
are the smoothing coefficients $A_{i,j}, B_{i,j}$. 

\begin{example}
The birack $X$ specified by the operation tables
\[
\begin{array}{r|rrr}
\utr & 1 & 2 & 3 \\\hline
1 & 1 & 3 & 1 \\
2 & 2 & 2 & 2 \\
3 & 3 & 1 & 3 \\
\end{array}
\quad
\begin{array}{r|rrr}
\otr & 1 & 2 & 3 \\\hline
1 & 3 & 3 & 3 \\
2 & 2 & 2 & 2 \\
3 & 1 & 1 & 1 \\
\end{array}
\]
has birack brackets over $\mathbb{Z}_5$ including
\[
\left[\begin{array}{rrr|rrr}
1 & 1 & 3 & 4 & 4 & 2 \\
3 & 2 & 4 & 2 & 3 & 1 \\
1 & 2 & 3 & 4 & 3 & 2 
\end{array}\right]
\]
and over $\mathbb{C}$ including
\[
\left[\begin{array}{rrr|rrr}
1 & -i & -i & -1 & i & i \\
-i & 1 & -1 & i & -1 & 1 \\
1 & 1 & -i & -1 & -1 & i 
\end{array}\right].
\]
\end{example}

\begin{corollary}
Let $L$ be an oriented link diagram, $X$ a finite birack and $\beta$ a birack
bracket with coefficients in a commutative unital ring $R$. Then the multiset 
of $\beta$ values over the birack homset, denoted $\Phi_X^{\beta,M}(L)$,
is an invariant of oriented framed knots and links.
\end{corollary}

\begin{definition}
We can convert the birack bracket multiset invariant into a polynomial for
ease of comparison. We define the \textit{birack bracket polynomial} 
associated to a birack $X$ and birack bracket $\beta$ as 
\[\Phi_X^{\beta}(L)=\sum_{x\in \Phi_X^{\beta,M}(L)} u^x.\]
\end{definition}

\begin{example}\label{E4} Consider the birack over the set 
$X=\{1, 2, 3\}$ with operation tables given by 
\[
\begin{array}{r|rrr}
\utr & 1 & 2 & 3 \\ \hline
1 & 2 & 2 & 2  \\
2 & 3 & 3 & 3  \\
3 & 1 & 1 & 1  \\
\end{array}
\quad
\begin{array}{r|rrr}
\otr & 1 & 2 & 3 \\ \hline
1 & 3 & 3 & 3    \\
2 & 1 & 1 & 1     \\
3 & 2 & 2 & 2      
\end{array}.
\]

One birack bracket with coefficients in $\mathbb {Z}_{5}$ 
as computed with our \texttt{python} code is given by
\[
\left[
\begin{array}{*{3}{c}|*{3}{c}}
1 & 2 & 1   & 1 & 2 & 1      \\
3 & 1 & 3   & 3 & 1 & 3        \\
3 & 1 & 3   & 3 & 1 & 3      
\end{array}
\right]
\]
with cycle value 
\[
    \delta= -A_{x,y} B_{x,y}^{-1}--A_{x,y}^{-1} B_{x,y}
    = -(1)(1)-(1)(1) = -2 = 3
\]
which is the same for all values of $x$ and $y$. 

Applying the birack bracket skein relation to obtain the
state sum for the framed trefoil knot with writhe 3, we have:
\[\includegraphics{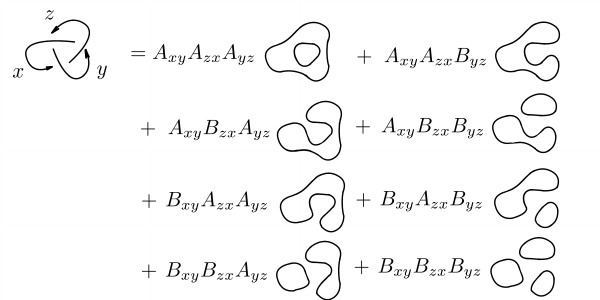}\]
yielding
\begin{eqnarray*}
\beta(f) & = & A_{xy}A_{zx}A_{yz}\delta^2+  
A_{xy}A_{zx}B_{yz}\delta+ 
A_{xy}B_{zx}A_{yz}\delta+  
A_{xy}B_{zx}B_{yz}\delta^2 \\
& &
+B_{xy}A_{zx}A_{yz}\delta+  
B_{xy}A_{zx}B_{yz}\delta^2+ 
B_{xy}B_{zx}A_{yz}\delta^2+ 
B_{xy}B_{zx}B_{yz}\delta^3
\end{eqnarray*}

Then substituting in the values of $A_{xy}, B_{xy}$ etc. for each of the
three colorings, we have
%The birack bracket multiset can be calculated through enumerating the smoothing on the 3 different colorings under $\mathbb{Z}_{5}$: 
\[
\begin{array}{ccc|c}   
      x & y & z & \phi \\
      \hline
      2 & 3 & 1 & (3)(2)(3)(3^2)\times 4 + (3)(2)(3)(3) \times3 +(3)(2)(3) (3^2) = 3 + 2 +1 = 1 \\
       3 & 1 & 2 &  (2)(3)(3)(3^2)\times 4 + (2)(3)(3)(3) \times3 +(2)(3) (3) (3^2) = 3 + 2 +1 = 1\\
       1 & 2 & 3 & (3)(3)(2)(3^2)\times 4 + (3)(3)(2)(3) \times3 +(3) (3)(2) (3^2) = 3 + 2 +1 = 1   
    \end{array}
\]

The resulting birack bracket multiset of the trefoil knot is 
$\Phi_{X}^{\beta,M} (3_1)$ is $\{1,1,1\}$, which in polynomial form is $3u$.
\end{example}

\begin{remark}
The framed trefoil with writhe 4 has no colorings by the birack in Example 
\ref{E4} and hence $\Phi_{X}^{\beta,M}(L) =\emptyset$ which in polynomial form is 
$0$. In particular the birack bracket invariant can distinguish different
framings of the same framed knot.
\end{remark}

\begin{example} \label{ex:l1}
Let $X$ be the birack given by the operation tables
\[
\begin{array}{r|rrr}
\utr & 1 &2 & 3 \\ \hline
1 & 2 & 2 & 1 \\
2 & 1 & 1 & 2 \\
3 & 3 & 3 & 3 \\
\end{array}
\quad
\begin{array}{r|rrr}
\utr & 1 &2 & 3 \\ \hline
1 & 1 & 1 & 2 \\
2 & 2 & 2 & 1 \\
3 & 3 & 3 & 3 \\
\end{array}.
\]
We compute via \texttt{python} that the coefficient matrix
\[
\left[\begin{array}{rrr|rrr}
1 & 1 & -1 & -1 & -1 & 1 \\
-1 & -1 & -1 & 1 & 1 & 1 \\
i & -i & 1 & -i & i & 1
\end{array}\right]
\]
defines a birack bracket on $X$ with coefficients in $\mathbb{C}$.
Then the framed Hopf link $L$ 
\[\includegraphics{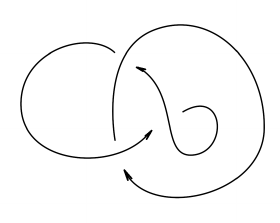}\]
has birack bracket polynomial
\[\Phi_X^{\beta}(L)=4u^{4i}+5u^4.\]
%>>> X[1]
%(((2, 2, 1), (1, 1, 2), (3, 3, 3)), ((1, 1, 2), (2, 2, 1), (3, 3, 3)))
%>>>> bbquiversg([[-1,2,-3.5,3.5],[1,-2,-4.5,4.5]],X[1],[[[1,1,-1],[-1,-1,-1],[I,-I,1]],[[-1,-1,1],[1,1,1],[-I,I,-1]]],[(3,3,3),(2,1,3)])
%[[1, 4.0*I, 9], [1, 4.0*I, 6], [2, 4.00000000000000, 9], [2, 4.00000000000000, 5], [3, 4.00000000000000, 9], [3, 4.00000000000000, 4], [4, 4.00000000000000, 9], [4, 4.00000000000000, 3], [5, 4.00000000000000, 9], [5, 4.00000000000000, 2], [6, 4.0*I, 9], [6, 4.0*I, 1], [7, 4.0*I, 9], [7, 4.0*I, 8], [8, 4.0*I, 9], [8, 4.0*I, 7], [9, 4.00000000000000, 9], [9, 4.00000000000000, 9]]
\end{example}

\section{\large\textbf{Birack Bracket Quivers}}\label{BBQ}

We begin this Section with a definition generalizing a definition in \cite{FN}.

\begin{definition}
Let $X$ be a finite birack, $S\subseteq \mathrm{Hom}(X,X)$ a set of birack 
endomorphisms and $L$ an oriented framed link. Then the \textit{birack 
coloring quiver} of $L$ with respect to $X$ and $S$, denoted 
$\mathcal{BRQ}_{X,S}(L)$, is the directed graph with a vertex for each 
birack coloring $v\in\mathrm{Hom}(\mathcal{BR}(L),X)$ 
and a directed edge from vertex $v_1$ to vertex $v_2$ whenever $v_2=v_1\sigma$
for each endomorphism $\sigma\in S$.
\end{definition}

This definition is motivated by the observation that if $\sigma:X\to X$ is
a birack endomorphism and $v_1$ is a valid birack coloring of a diagram of $L$,
then applying $\sigma$ to each color in $v_1$ yields another valid coloring,
$v_2=v_1\sigma$.
\[\includegraphics{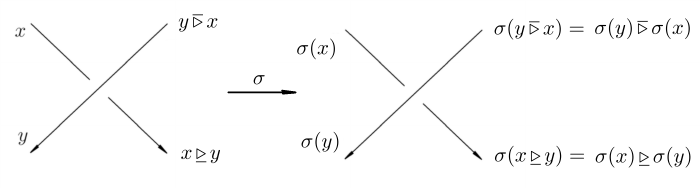}\]

As in previous cases such as \cite{CCN,CN,FN} etc., we observe that since the
ingredients for $\mathcal{BRQ}_{X,S}(L)$ are all invariant under framed 
Reidemeister moves, so is $\mathcal{BRQ}_{X,s}(L)$. That is, we have:

\begin{theorem}
For any finite birack $X$ and $S\subseteq\mathrm{Hom}(X,X)$, the quiver
$\mathcal{Q}_{X,S}(L)$ is an invariant of framed isotopy.
\end{theorem}

Next, following earlier work such as \cite{FN}, we come to our primary new 
definition.

\begin{definition}\label{def:main}
Let $X$ be a finite birack, $S\subseteq\mathrm{Hom}(X,X)$ a set of birack 
endomorphisms, $\beta$ a birack bracket and $L$ an oriented framed link. The
\textit{birack bracket quiver} of $L$ with respect to $X,S$ and $\beta$, 
denoted $\mathcal{BRQ}_{X,\beta}^{S}(L)$, is the birack coloring quiver of
$L$ with respect to $X$ and $S$ with vertices $v$ weighted with the birack 
bracket values $\beta(v)$ associated to the colorings $v$. 
\end{definition}

A finite quiver is a small category with vertices as objects and directed paths
as morphisms; hence, Definition \ref{def:main} establishes a family of 
categorifications of the birack bracket multiset and polynomial 
invariants from the previous section.

The isomorphism class of $\mathcal{BRQ}_{X,S}^{\beta}$ is an invariant of framed
oriented knot and links and can thus be used to distinguish framed oriented
knots and links. Since these quivers become complex quite quickly, it is
also useful to extract some simpler invariants via decategorification.

\begin{definition}
Given a finite birack $X$, a framed knot $K$ and a birack bracket $\beta$
with coefficients in a commutative unital ring $R$ and 
$S\subset \mathrm{Hom}(X,X)$ a set of endomorphisms of $X$.
Then the birack bracket quiver $\mathcal{BRQ}_{X,S}^{\beta} (K)$ has 
vertex set $V$ and edge set $E$. Let $P$ the set of maximal non-repeating
paths (i.e., ordered sequences of edges with each edge terminating at the 
initial vertex of the next edge such that no edge is included more than 
once and such that the path is not a sub-path of any other path)
in the quiver. We then decategorify this birack bracket 
quiver $\mathcal{BRQ}_{X,S}^{\beta} (K )$ to obtain the 
\textit{birack bracket in-degree polynomial} 
    \[
    \Phi_{X, \beta}^{S, \mathrm{deg}_+} (K) = \sum_{v \in V} u^{\beta (v)} v^{deg+(v)},
    \]
the \textit{birack bracket two-variable polynomial} 
    \[
    \Phi_{X, \beta}^{S, 2} (K) = \sum_{e \in E} s^{\beta (s(e))} t^{\beta (t (e)) }
    \]
 where $s(e)$ and $t(e)$ represent the source and tail of edge $e \in E$, and
the \textit{birack bracket maximal path polynomial}
    \[
    \Phi_{X, \beta}^{S, MP} (K) = \sum_{p\in P} x^{\sum_{v\in p}\beta (v)} y^{|p|}
    \]
where $v\in p$ means $v$ is a vertex in the path $p$ and $|p|$ is the 
length of the path $p$.
\end{definition}

\begin{example}
Recall the three different colorings of the trefoil knot $3_1$ by the constant
action birack structure on $X = \{ 1, 2 ,3 \} $ in \ref{ex:31homset}.
Let the endomorphism $S =\{\phi\}\subset \mathrm{Hom}(X,X)$ where 
$\phi(1) = 2, \phi(2) =3$ and $\phi(3) =1$. We can specify this endomoprhism
with a list of output values, i.e. $\sigma=[2,3,1]$.
Then we have birack coloring quiver 
\[\includegraphics{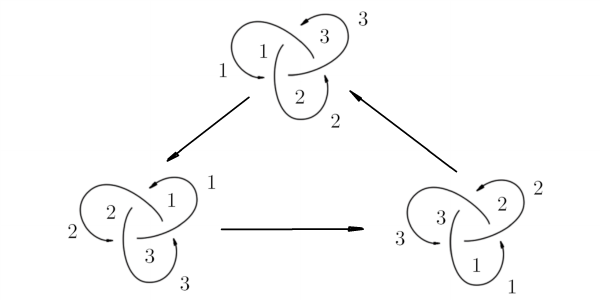}.\]

The birack coloring quiver then replaces each coloring with a vertex weighted 
by its birack bracket value; in this case we have
\[\includegraphics{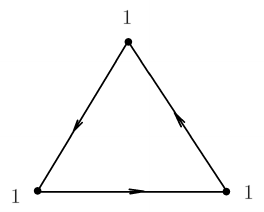}.\]
The framed knot invariant polynomials obtained from decategorification of 
this birack bracket quiver are then
    \[
    \Phi_{X, \beta}^{S, deg+} (3_{1}) = 3uv
    \]
    and 
    \[
    \Phi_{X, \beta}^{S, 2} (3_{1}) = 3 st.
    \]
\end{example}

\begin{example}
The birack in Example \ref{ex:l1} has endomorphism set 
$\mathrm{Hom}(X,X)=\{\sigma_1,\sigma_2,\sigma_3\}=\{[1,2,3],[3,3,3],[2,1,3]\}$,
i.e., where $\sigma_1$ is the identity map, $\sigma_2$ is the constant map 
sending everything to 3, and $\sigma_3$ fixes 3 while swapping 1 and 2. 
Then the birack bracket quiver of the framed Hopf link
in Example \ref{ex:l1} with respect to the birack bracket in the same example
and endomorphism set $S=\{\sigma_2,\sigma_3\}$ is
\[\includegraphics{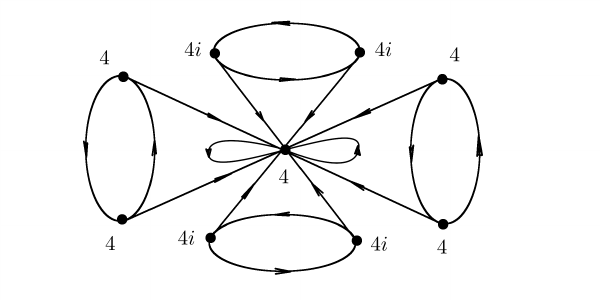}.\]

%>>>> bbquiversg([[-1,2,-3.5,3.5],[1,-2,-4.5,4.5]],X[1],[[[1,1,-1],[-1,-1,-1],[I,-I,1]],[[-1,-1,1],[1,1,1],[-I,I,-1]]],[(3,3,3),(2,1,3)])
%[[1, 4.0*I, 9], [1, 4.0*I, 6], [2, 4.00000000000000, 9], [2, 4.00000000000000, 5], [3, 4.00000000000000, 9], [3, 4.00000000000000, 4], [4, 4.00000000000000, 9], [4, 4.00000000000000, 3], [5, 4.00000000000000, 9], [5, 4.00000000000000, 2], [6, 4.0*I, 9], [6, 4.0*I, 1], [7, 4.0*I, 9], [7, 4.0*I, 8], [8, 4.0*I, 9], [8, 4.0*I, 7], [9, 4.00000000000000, 9], [9, 4.00000000000000, 9]]

\end{example}

\begin{example}
\label{E6} 
Let $X$ be the birack defined by the operation tables
\[
    \begin{array}{c|c c c c}
      \utr & 1 & 2 & 3 & 4\\
      \hline
      1 & 2 & 4 & 2 & 4 \\
      2 & 1 & 3 & 1 & 3 \\
      3 & 4 & 2 & 4 & 2 \\
      4 & 3 & 1 & 3 & 1
    \end{array}
\quad
    \begin{array}{c|c c c c}
      \otr & 1 & 2 & 3 & 4 \\
      \hline
      1 & 4 & 2 & 4 & 2 \\
      2 & 3 & 1 & 3 & 1 \\
      3 & 2 & 4 & 2 & 4 \\  
      4 & 1 & 3 & 1 & 3 \\  
    \end{array}
\]
Then one birack bracket found with our \texttt{python} code over 
$\mathbb{Z}_5$ is given by
    \[
        \left[
        \begin{array}{cccc|cccc}
        1 & 1 & 2 & 1   & 2 & 2 & 4 & 2     \\
        2 & 2 & 2 & 1   & 4 & 4 & 4 & 2     \\
        1 & 1 & 1 & 3   & 2 & 2 & 2 & 1     \\
        2 & 2 & 4 & 2   & 4 & 4 & 3 & 4 
        \end{array}
        \right].
        \]

The set of endomorphisms of this birack $X$ is  
\[\mathrm{Hom}(X,X)=\{[4,2,3,1],[4,3,2,1],[1,2,3,4],[1,3,2,4]\}.\]
%\{\phi: X \to X \ |\ (1 2 3 4 ) \longmapsto (4 2 3 1), (4 3 2 1), ( 1 2 3 4), (1 3 2 4)\}.\]

With these, the birack bracket quiver invariants for various blackboard
framed links from the table at \cite{KA} are:

\[
    \begin{array}{r|l}   
       \Phi_{X, \beta}^{S, 2} (L) & L \\
      \hline
       16u^4 v^4 & L2a1, L6a2, L6a3, L7a5, L7a6 \\
       16u^4 v^4 +32 u^3 v^4 +16uv^4 & L4a1 \\
       32u^4v^4  &  L5a1, L7a1, L7a3, L7a4, L7n2 \\
       32u^4v^{4} + 32u^3v^{4}  & L6a1\\
       32u^4v^4+32u^3v^4+96u^2v^4+96uv^4  & L6a4 \\
       32u^4v^4+32u^3v^4+96u^2v^4   & L6a5, L6n1\\
       16u^4v^4+16u^3v^4 & L7a2\\
        32u^3v^4 + 32u v^4 & L7a7 \\
       16u^4v^4 + 16 u^2v^4   & L7n1\\
    \end{array}.
\]

With the same birack and birack bracket, the birack bracket quiver 
invariants for various blackboard framed virtual knots from the table 
at \cite{KA} are: 

\[    \begin{array}{r|l}   
       \Phi_{X, \beta}^{S, 2} (VK) & VK \\
      \hline
       8u^4v^4+8u^2v^4 & 2.1, 4.21, 4.24, 4.28, 4.36  \\
       8u^2v^4 & 3.1, 3.3, 3.3, 3.4, 3.5, 3.6, 3.7 \\
       8u^3v^4 +8u^2v^4 & 4.1, 4.3, 4.9, 4.15, 4.25, 4.28, 4.29, 4.37\\
       16u^2v^4  & 4.2, 4.6, 4.8, 4.12, 4.13, 4.14\\
       8u^2v^4 +8uv^4 &  4.4, 4.5, 4.10, 4.11, 4.16, 4.18, 4.23, 4.27, 4.30, 4.31, 4.33, 4.38, 4.39\\
       16u^2v^4  &  4.2, 4.17, 4.19, 4.20, 4.22, 4.26, 4.32, 4.34, 4.35, 4.40
\end{array}\]
\end{example}

\begin{example}\label{E10}
Continuing with the birack and birack bracket from Example \ref{E6},
the following example shows our new decategorification invariants 
distinguishing different framings of the same knot, in this case the Figure 8
knot $4_1$ with framings from $-4$ to $4$. We illustrate the
cases of framing numbers $-1$, $0$ and $1$:
\[
        \includegraphics{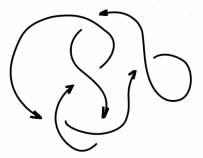}
        \includegraphics{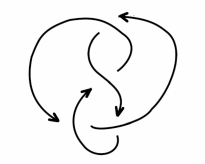}
        \includegraphics{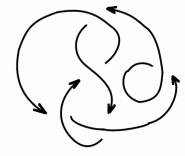}
\]

Then we compute invariant values in the table.
\[    
\begin{array}{r|l}   
       \Phi_{X, \beta}^{S, 2} (4_1) & 4_1 \mathrm{\ with\ framing\ number\ 
from\ -4 \ to\ 4}\\
      \hline
       16u^2v^4 &  4_1^0 \\
       8u^2v^4 &  4_1^1, 4_1^3, 4_1^{-1}, 4_1^{-3} \\
       8u^4v^4+8u^2v^4 &  4_1^2 \\
       8u^3v^4 + 8u^2v^4 &  4_1^4, 4_1^{-4} \\
       8u^2v^4 + 8uv^4&  4_1^{-2}
\end{array}\] 
where the superscript denotes framing number from $-4$ to $4$.
\end{example}

We conclude this section with another example.

\begin{example}
Let $X$ be the birack given by the operation tables
\[
\begin{array}{r|rrr}
\utr & 1 & 2 & 3 \\ \hline
1 & 2 & 2 & 2 \\
2 & 1 & 1 & 1 \\
3 & 3 & 3 & 3
\end{array}
\quad
\begin{array}{r|rrr}
\otr & 1 & 2 & 3 \\ \hline
1 & 1 & 1 & 2 \\
2 & 2 & 2 & 1 \\
3 & 3 & 3 & 3
\end{array}.
\]
Then we compute via \textit{python} that the block matrix
\[
\beta=\left[
\begin{array}{rrr|rrr}
1 & 1 & i & -q & -q & iq \\
i & i & -1 & -iq & -iq & q\\
i & 1 & 1 & -iq & -q & -q 
\end{array}\right]
\]
defines a birack bracket over $\mathbb{C}[q^{\pm 1}]$ and the maps
\[\sigma_1=[3,3,3],\quad \sigma_2=[2,1,3]\]
are endomorphisms of $X$. Then we compute that the framed virtual links
\[\includegraphics{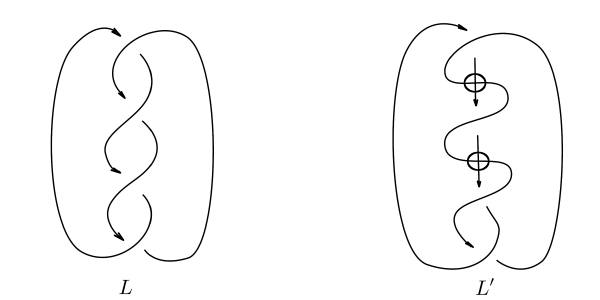}\]
are distinguished by the birack maximal path polynomial with
\[\Phi_{X,S}^{\beta,MP}(L_1)=8x^{6q^8+6q^2+6+6q^{-2}}y^5+8y^5\]
and
\[\Phi_{X,S}^{\beta,MP}(L_2)=4x^{(3-3i)q^8+(3-3i)q^2+3-3i+(3-3i)q^{-2}}y^5
+4x^{(3+3i)q^8+(3+3i)q^2+3+3i+(3+3i)q^{-2}}y^5.
\]
%>> bqcmaxpathg(X[3],[[1,-2,3,-4],[-1,2,-3,4]],[(3,3,3),(2,1,3)],Y3)
%8*x**(6*q**8 + 6*q**2 + 6 + 6/q**2)*y**5 + 8*y**5
%>>> bqcmaxpathg(X[3],[[1,2,3,-4],[-1,-2,-3,4]],[(3,3,3),(2,1,3)],Y3)
%4*x**(3*q**8 + 3*q**2 + 3 + 3*I*(-q**10 - q**4 - q**2 - 1)/q**2 + 3/q**2)*y**5 + 4*x**(3*q**8 + 3*q**2 + 3 + 3*I*(q**10 + q**4 + q**2 + 1)/q**2 + 3/q**2)*y**5
%>>> jones([[1,2,3,-4],[-1,-2,-3,4]])
%a**8 + a**4 + 1 + a**(-12)
%>>> jones([[1,-2,3,-4],[-1,2,-3,4]])
%a**8 + a**4 + 1 + a**(-12)
%>>> X[3]
%(((2, 2, 2), (1, 1, 1), (3, 3, 3)), ((1, 1, 2), (2, 2, 1), (3, 3, 3)))
%>>> Y3
%[[[1, 1, I], [I, I, -1], [I, 1, 1]], [[-q, -q, -I*q], [-I*q, -I*q, q], [-I*q, -q, -q]]]

We note that these two virtual links have the same Jones polynomial
value of $q^8+q^4+1q^{-12}$, so this example shows that the birack 
bracket quiver (and hence its decategorifications) are not determined by 
the Jones polynomial.
\end{example}

\section{\large\textbf{Questions}}\label{Q}

It is important to note that the examples in this paper are, generally speaking,
``toy'' examples, using small-cardinality biracks and typically finite 
rings since these are easy to compute; they are not to be taken to represent 
the full power of this infinite family of invariants. In particular, larger
biracks and larger finite or infinite coefficient rings should provide more
powerful invariants.

In biquandle bracket theory, as in the Kauffman bracket polynomial,
we can adjust the value of the bracket of any framing of a knot or link
back to the value of the zero framing by multiplying by an appropriate power
of $w$. In birack bracket theory this no longer works for the simple reason 
that different framings can have different numbers of colorings, so adjusting
a given coloring to a given writhe may not make sense. This is reflected in
the fact that 
%Additionally, we discovered one interesting detail regarding the proposed framing Reidemeister I move. The smoothing of the kink map under birack bracket introduces another unit value, which we terms $w $. Compared to the other unit identity $\delta$ defined by the unknot, 
the value of $w_x$ is not always invariant over the entire birack bracket. 
What do these values of $w_x$ imply about the algebraic structure of 
generic birack brackets? What global information about chosen biracks can 
a combination of $w_x$ values over corresponding birack brackets describe?

We are curious about the possibility of functors between framed 
knots/links and unframed knots/links.
Over chosen finite rings on which the birack brackets are defined, one can 
add different framing numbers as shown in Example \ref{E10}. 
The reverse process is 
to ``smooth out'' the framing information as a state sum of different framed 
knot/link states. The corresponding factors might be defined with respect to 
the $w$ values. We hope to resolve these issues in a future paper.

What information about a framed knot or link is contained in the geometric 
structure of the birack coloring quiver? We can observe that paths in the quiver
always terminate in a cycle since the out-degree of every vertex is always 
positive; similarly, endomorphisms which are automorphisms tend to produce
cycles in the quiver. What other characterizations can be made? 

Finally, we ask what is the algebraic structure of the set of birack brackets?
What operations can be used to combine brackets over a given birack to obtain
valid birack brackets? 

\bigskip

\bibliographystyle{abbrv}
\bibliography{sn-tt}

\noindent
\textsc{Department of Mathematical Sciences \\
Claremont McKenna College \\
850 Columbia Ave. \\
Claremont, CA 91711} 

\medskip

\noindent
\textsc{Department of Physics and Astronomy\\
Pomona College \\
333 N. College Way \\
Claremont, CA 91711}

\end{document}